\documentclass[11pt]{amsart}
\usepackage[margin=1.5in]{geometry}
\usepackage{mathrsfs}
\usepackage{amsfonts}
\usepackage{latexsym,epsfig}
\usepackage{mathabx}
\usepackage{amsmath, amscd, amsthm,amssymb}
\usepackage[pdftex]{color}
\usepackage{comment}
\usepackage{marginnote}
\usepackage[colorlinks,citecolor=blue,linkcolor=red]{hyperref}
\usepackage[nameinlink]{cleveref}
\usepackage{enumitem}
\usepackage{enumitem}

\newcommand{\ZZ}[0]{\mathbb{Z}}

\newcommand{\wt}{\widetilde}
\newcommand{\mc}{\mathcal}

\newcommand{\wh}{\widehat}

\newcommand{\mr}{\mathring}

\usepackage{todonotes}

\usepackage{hyperref}

\usepackage{graphicx}
\usepackage{pinlabel}

\newtheorem{theorem}{Theorem}[section]
\newtheorem{lemma}[theorem]{Lemma}
\newtheorem{proposition}[theorem]{Proposition}
\newtheorem{corollary}[theorem]{Corollary}

\newtheorem{claim}{Claim}
\newtheorem{definition}[theorem]{Definition}

\theoremstyle{definition}

\newtheorem{remark}[theorem]{Remark}
\newtheorem{convention}[theorem]{Convention}

\newcommand{\FF}{\mathcal{F}}

\newcommand{\Aut}{\mathrm{Aut}}

\newcommand{\orb}{\mathcal{O}}

\newcommand\tsim{\kern-.4em\sim}

\newcommand\ssm{\smallsetminus}

\renewcommand{\phi}{\varphi}
\renewcommand{\epsilon}{\varepsilon}

\title[Pseudo-Anosov orbit spaces via bifoliated planes]{A characterization of pseudo-Anosov orbit spaces via bifoliated planes}
\author{Kathryn Mann and Samuel J. Taylor}

\begin{document}

\begin{abstract}
We characterize the actions on bifoliated planes that arise as orbits spaces of transitive pseudo-Anosov flows on orientable closed $3$-manifold. We use branched covers, veering triangulations, and a new compactness criterion to extend previous work that handled the special case where the plane has no odd-prong singularities and the action preserves a leafwise orientation on the foliations.
\end{abstract}

\maketitle

\setcounter{tocdepth}{1}
\tableofcontents

\section{Introduction}
A {\em bifoliated plane} is a plane $\mc P$ together with a pair of topologically transverse foliations $\FF^h, \FF^v$, possibly with isolated prong singularities.  See \Cref{sec_prelim} for a formal description.  
We assume throughout this work that bifoliated planes have no {\em infinite horizontal rectangles}, that is, no properly embedded subset isomorphic to $[0,\infty) \times [0, 1]$ with its horizontal and vertical product foliation.  
When we say a group $G$ acts on $(\mc P, \FF^h, \FF^v)$, we mean that it acts by homeomorphisms of $\mc P$ preserving the two foliations.  

The main motivating example is the {\em orbit space} of a pseudo-Anosov flow on a compact $3$--manifold $M$.  This space is a bifoliated plane with an action of $\pi_1(M)$, which Barbot \cite{barbot1995caracterisation} showed completely determines the flow up to orbit equivalence.  Thus, such actions have become a key tool in the study of pseudo-Anosov flows.  However, Barbot's proof that the action determines the flow is non-constructive.  The {\em reconstruction problem} asks what dynamical conditions on a group action precisely characterize the orbit space actions of such pseudo-Anosov flows.  
This has been partially answered in recent work of Barthelm\'e--Fenley--Mann \cite{BFM25} and Baik--Wu--Zhao \cite{baik2024reconstruction}, however both papers needed to assume a leafwise orientation condition that excluded planes from having odd degree prongs.   

In this work, we treat the general case, giving a complete dynamical characterization of actions 
that come from orbit spaces of transitive flows on compact, orientable\footnote{for simplicity, we restrict to the case where manifolds are orientable, however the same tools should apply in the nonorientable setting}
$3$-manifolds, and describe a reconstruction process independent from Barbot's theorem.

\begin{theorem} \label{th:main}
Suppose a group $G$ acts on a
bifoliated plane $(\mc P, \FF^h, \FF^v)$. There exists a closed, oriented 3-manifold $M$ with $\pi_1(M) =G$ supporting a transitive pseudo-Anosov flow $\phi$ whose orbit space $\orb_\phi$ is equivariantly isomorphic to $(\mc P, \FF^h, \FF^v)$ if and only if $G$ is torsion-free, the action $G \curvearrowright \mc P$ is transitive, preserves an orientation, and has:
\begin{itemize}
\item the closing property,
\item uniformly hyperbolic fixed points, and 
\item the finite rectangle condition.
\end{itemize}
\end{theorem}

Here, \emph{transitivity} of the action $G \curvearrowright \mc P$ refers to topological transitivity, i.e. there exists a point whose $G$--orbit is dense.

\smallskip
The closing property, uniformly hyperbolic fixed points, and finite rectangle condition are dynamical conditions capturing properties satisfied by all pseudo-Anosov flows, inspired by (weak versions of) the closing lemma, hyperbolicity, and having a finite cover by flow-boxes.  See \Cref{sec_compactness} for the definitions. 
The torsion-free condition on $G$ is necessary to ensure $M$ is a manifold, because there are orbifold groups with actions satisfying the other conditions.

\smallskip
There is another closely related context where reconstruction problems for bifoliated planes have received considerable recent attention, and this is the setting of loom spaces and veering triangulations.  As part of their program to give a precise correspondence between pseudo-Anosov flow without perfect fits and certain ideal triangulations of $3$-manifolds, called \emph{veering triangulations}, Segerman and Schleimer explain how to start with an action on a loom space (a special type of singularity-free bifoliated plane), build a {veering triangulation} of a $3$-manifold with torus cusps, and from there produce pseudo-Anosov flows without perfect fits on a Dehn filled closed manifold (\cite{schleimer2024loom, schleimer2019veering, SS3}). See also \cite[Section 6]{BFM25} where a different approach is given to construct an expansive flow directly on the cusped manifold.

Our work relates these two different frameworks, using veering triangulations as a tool towards the proof of \Cref{th:main}, and describing the relationship between the reconstruction strategy of \cite{BFM25}, which builds ``coordinates" for a 3-manifold $M$ out of data of a bifoliated plane (under a leafwise orientability condition), and of \cite{schleimer2024loom, LMT21} which builds an ideal triangulation of a $3$-manifold out of {\em maximal rectangles} in the bifoliated plane. 

To relate these two approaches, we also introduce the tool of branched covers of bifoliated planes. This can be interpreted as an approach to Fried surgery, a flow preserving version of Dehn surgery along closed orbits \cite{fried1983transitive}, solely from the perspective of the associated action on the orbit space of the flow. Since Fried surgery plays an important role in the general theory, we believe this perspective will have future applications.

\subsection{Outline and strategy of proof} 
\Cref{sec_prelim} describes a framework to study branched covers of bifoliated planes, defines the key properties from \Cref{th:main}, and shows that these pass naturally to $G$--compatible normal branched covers and to finite index subgroups.  

In \Cref{sec_compactness}, without assuming transitivity, we show the following:  
\begin{theorem} \label{thm_compactness_no_transitive}
Suppose the torsion-free group $G$ acts on bifoliated plane $(\mc P, \FF^h, \FF^v)$ preserving orientations along the leaves.  Such an action comes from 
a  pseudo-Anosov flow on a compact 3-manifold if and only if it satisfies: 
\begin{itemize}
\item the closing property,
\item uniformly hyperbolic fixed points, and 
\item the finite rectangle condition.
\end{itemize}  
\end{theorem} 
This builds on the work of \cite{BFM25}, which did not give a criterion for compactness, but instead discussed expansive flows on (possibly noncompact) 3-manifolds. 

In \Cref{sec:veer}, we relate the framework of \cite{BFM25} to veering triangulations, 
and show how to build a $3$--manifold with $\pi_1(M) = G$ together with a transitive pseudo-Anosov flow
 out of the data of a transitive action of 
$G$ on a bifoliated plane.  Here we assume that the action preserves an orientation on $\mc P$ but do not require a leafwise orientation.  

In \Cref{sec:same_action}, we describe a correspondence between points in the bifoliated plane and transverse paths through a veering triangulation (statement of \Cref{prop:recovering}), which allows us to show that the flow constructed in \Cref{sec:veer} has orbit space equivariantly isomorphic to  $(\mc P, \FF^h, \FF^v)$. This establishes the `if' direction of \Cref{th:main}. The easier `only if' direction is a consequence of a more general statement given in \Cref{prop:only_if}.

\subsection*{Acknowledgements}
We thank Thomas Barthelm\'e and Chi Cheuk Tsang for helpful comments on an earlier draft of the paper.

This research was supported by the National Science Foundation under Grant No. DMS--2424139, while the authors were in residence at the Simons Laufer Mathematical Sciences Institute in Berkeley, California, during the Spring 2026 semester.

Mann was also partially supported by NSF grant  DMS-2505228.
Taylor was also partially supported by NSF grant DMS--2503113 and the Simons Foundation.

\section{Dynamics on bifoliated planes and branched covers} \label{sec_prelim} 

In this section we introduce branched covers of bifoliated planes and discuss how various properties of actions on bifoliated planes can lift through branched covers. 

We begin by setting some conventions and standing assumptions to be used throughout this work.  
A {\em bifoliated plane} is a topological plane $\mc P$ with two transverse, possibly singular, foliations $\FF^h$ and $\FF^v$.  We call $\FF^h$ the {\em horizontal foliation} and $\FF^v$ the {\em vertical foliation}.  Singularities are required to be of {\em prong type}; this means they are locally modeled on the image of the horizontal and vertical coordinate foliations of $\mathbb{R}^2 \cong \mathbb{C}$ at 0 under the semi-branched cover $z \mapsto z^{k/2}$ for some $k \geq 2$, and each leaf may contain at most one singular point.  An \emph{automorphism} of $\mc P$ is a homeomorphism of $\mc P$ preserving each foliation and we denote the group of automorphisms by $\Aut(\mc P)$.

We also make the standing assumption that $(\mc P, \FF^h, \FF^v)$ has no horizontal
{\em infinite product regions}, that is, there is no subset of $\mc P$ properly homeomorphic to $[0, \infty) \times [0, 1]$ under a map sending leaves $\{p\} \times [0,1]$ to vertical leaves and $[0, \infty) \times \{p \} $ to horizontal leaves. Moving forward, we will often refer to the bifoliated plane $(\mc P, \FF^h, \FF^v)$ by simply $\mc P$. 

The $\FF^{h/v}$ {\em saturation} of a subset $A$, denoted $\FF^{h/v}(A)$, is the union of all $\FF^{h/v}$-leaves intersecting $A$.  By convention, when we write $\FF^{h/v}$ in a statement we mean that the statement is true when all instances of $\FF^{h/v}$ are replaced by $\FF^h$, and when all are replaced by $\FF^v$.

\smallskip
Finally, we will make frequent use of {\em rectangles} in $\mc P$.  For our purposes, a closed rectangle is a compact subset $R \subset \mc P$ isomorphic (as a bifoliated set) to $[0, 1]^2$ with its coordinate foliations.    
Throughout the paper, all rectangles will be closed unless otherwise stated.

\subsection{Branched covers}
\label{sec:branched}
Suppose that $\mc Q$ is a bifoliated plane and that a subgroup $H \le \Aut(\mc Q)$ acts properly discontinuously on $\mc Q$
so that the quotient $\mc P = \mc Q / H$ is itself a bifoliated plane, when considered with the foliations induced from $\mc Q$.
Then we call the quotient map $\pi \colon \mc Q \to \mc P$ a \textit{normal branched cover} of bifoliated planes. The branch points of $\mc Q$ are exactly the points fixed by nontrivial elements of $H$. 

If $G \curvearrowright \mc P$ is a group acting on $\mc P$ by automorphisms, we say that the normal branched cover $\pi \colon \mc Q \to \mc P$ is \emph{$G$--compatible} if each $g \in G$ admits a lift to an automorphism of $\mc Q$. In this case, the coset $gH$ accounts for the full set of lifts of $g$ to $\mc Q$.
Hence, if $\wh G$ denotes the group of all lifts of $g \in G$, there is an exact sequence:
\[
1 \to H \to \wh G \to G \to 1.
\]

Since $\mc P = \mc Q / H$ inherits a natural orbifold structure, the $G$-compatible branched covers of $\mc P$ arise in the following way. Let $G \curvearrowright \mc P$ be an action by automorphisms. Let $O$ be any $G$--invariant, closed, discrete subset of $\mc P$, and assign a positive integer to each point of $O$ so that this labeling of points in $O$ is invariant under the action of $G$. Now if we consider $\mc P$ as an orbifold with cone points at $O$ such that the order of $o \in O$ is given by its label, then we can take $\pi \colon \mc Q \to \mc P$ to be the associated universal orbifold cover. This is a normal orbifold cover with deck group $H$, and we can pull the foliations of $\mc P$ back to foliations on $\mc Q$. 
With this structure, $\pi \colon \mc Q \to \mc P$ is a normal branched cover of bifoliated planes. Moreover, since $G$ acts on $\mc P$ preserving its orbifold structure by construction, each $g \in G$ lifts to $\mc Q$ and so the map $\pi \colon \mc Q \to \mc P$ is also $G$--compatible. 

\begin{remark}
For those unfamiliar with orbifold covers, $\pi \colon \mc Q \to \mc P$ can also be constructed as follows. Consider the punctured plane $\mc P \ssm O$ and let $K$ be the normal subgroup of $\pi_1(\mc P \ssm O)$ generate by, for each $o \in O$ with label $k$, the loops $\gamma_o^k$. Here, $\gamma_o$ is the boundary of a disk in $\mc P$ with $o$ in its interior and not meeting any other points of $O$. Then the honest cover $\pi \colon \mc Q \ssm \pi^{-1}(O) \to \mc P \ssm O$ induced by the orbifold cover $\pi$ 
 is exactly the  
 covering space of $\mc P \ssm O$ associated to the subgroup $K$ with deck group $H = \pi_1(\mc P \ssm O) /K$. From this description, it is easy to construct $\pi \colon \mc Q \to \mc P$ directly.
\end{remark}

\subsection{Finite rectangles, uniform hyperbolicity, and closing}  
We begin with a fundamental relation between rectangles in $\mc P$.

\begin{definition}[Above/below in $\mc P$] \label{def:above} 
Let $R_1$ and $R_2$ be rectangles in $\mc P$.  We say $R_2$ is \emph{above} $R_1$ if the $\FF^v$-saturation
$R_2$ is strictly contained in the $\FF^v$-saturation of $R_1$, and the $\FF^h$-saturation of $R_1$ is strictly contained in that of $R_2$.  
Say $R_1$ is \emph{below} $R_2$ if $R_2$ is above $R_1$. 
\end{definition} 

Note that ``above" and ``below'' are transitive relations: if $R_3$ is above $R_2$ and $R_2$ above $R_2$, then $R_3$ is above $R_1$. As an alternative equivalent characterization, the rectangle $R_2$ is above $R_1$ if and only if each vertical segment of $R_2$ properly contains a vertical segment of $R_1$ and each horizontal segment of $R_1$ properly contains a horizontal segment of $R_2$.

The motivation for the definition is as follows:  
Suppose $\wt \phi$ is a lift of a pseudo-Anosov flow on a compact manifold $M$ to the universal cover $\wt M \cong \mathbb{R}^3$. We think of $\wt \phi$ as being oriented ``upwards."
Suppose $R_1$ is a local transversal to $\wt \phi$,  foliated by the unstable ($\FF^h$) and stable ($\FF^v$) foliations, and $R_2$ is a isometric image of $R_1$ situated above $R_1$ on $\wt M$, so flowing forward in time gives a (partially defined) map from $R_1$ to $R_2$.   The contracting property of $\phi$ along stable leaves, and contraction in the past along unstable leaves, implies that the projection of $R_1$ and $R_2$ to the orbit space will -- if $R_2$ and $R_1$ are centered about the same orbit, and $R_2$ is sufficiently far above $R_1$ in $\wt M$ -- be in the configuration described by \Cref{def:above}. 

 \begin{figure}
   \labellist 
  \small\hair 2pt
     \pinlabel $R_1$ at 100 100
     \pinlabel $R_2$ at 100 142
    \pinlabel $\wt M^3$ at 230 100
     \pinlabel $\mc P$ at 230 30
      \pinlabel $R_2$ at 60 7
     \pinlabel $R_1$ at 10 25
 \endlabellist
\includegraphics[width=7cm]{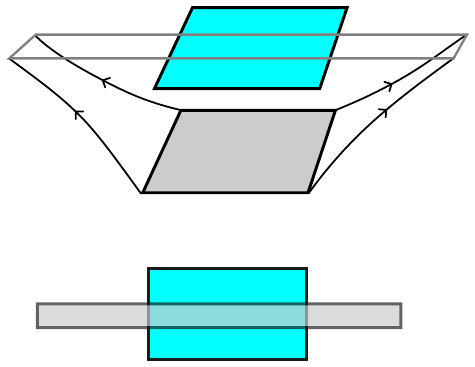}
\caption{Contraction along stable leaves means isometric translates of rectangles above or below have intersections in the orbit space $\mc P$ as in \Cref{def:above}}
 \label{fig_above} 
\end{figure}
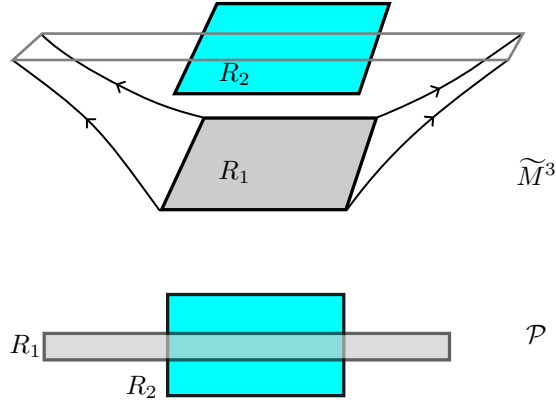

\begin{definition}[Finite rectangle condition] \label{def:rect}
Let $G$ be a group acting on $\mc P$ by automorphisms.   We say that this action has the \emph{finite rectangle condition} if there exists a finite family $\mc R$ of closed rectangles in $\mc P$ such that: 
\begin{itemize}
\item $\mc P = \bigcup_{g \in G; R \in \mc R} g R$, i.e. translates of rectangles in $\mc R$ cover $\mc P$, 
\item each $R \in \mc R$ is covered by a finite collection $a_i R_i$ with $R_i \in \mc R$ and $a_i \in G$ each of which is above  $R$, and 
\item each $R \in \mc R$ is covered by a finite collection $b_j R_j$ with $R_j \in \mc R$ and $b_j \in G$ each of which is below $R$.
\end{itemize}
\end{definition} 

In the context of flows, this is reminiscent of the existence of a Markov partition, but much weaker than that because we allow the rectangles in our cover to overlap.  It more closely parallels the simple fact that a compact manifold with a flow can be covered by finitely many flow-box neighborhoods.   Compare also Iakovoglou's definition of {\em Markovian family} \cite{iakovoglou2025markovian} and the compactness conditions given by Remark 7.7 and Construction 5.10 of \cite{baik2024reconstruction}.  

\smallskip
We recall also two conditions from \cite{BFM25}. 
These are orbit-space formulations of (weak versions of) the fact that only finitely many periodic orbits of bounded length can meet a compact set, and the property given by the Anosov closing lemma.  

First, recall that a point $x \in \mc P$ fixed by $g \in G$ is \emph{hyperbolic} if $g$ acts as a topological contraction on one of the leaves through $x$ and as an expansion on the other. The action has \emph{hyperbolic fixed points} if this is the case for any fixed point of a nontrivial $g \in G$. The following property is a mild strengthening of this condition.

\begin{definition}[Uniformly hyperbolic fixed points]  \label{def:uniformlyhyp}
The action $G \curvearrowright \mc P$ has \emph{uniformly hyperbolic fixed points} if fixed points are hyperbolic, 
and if for any compact rectangle $R$ and any sequence $(g_n)_{n\ge 1}$ from $G$, if $g_n(R)$ is always above $g_{n-1}(R)$ (resp. $g_n(R)$ below $g_{n-1}(R))$ then $\bigcap_n \FF^v(g_n(R))$ (resp. $\bigcap_n \FF^h(g_n(R))$) is a single leaf. 
\end{definition}

The condition that $\bigcap_n \FF^v(g_n(R))$ is a single leaf (under the hypothesis that $g_n(R)$ is always above $g_{n-1}(R)$)  is equivalent to the more local condition that $\bigcap_n g_n(R)$ is a vertical segment of the rectangle $R$. 

\begin{remark} \label{rem:uniform_hyp_difference}
This formulation is a priori slightly stronger than the condition given in \cite[Definition 1.9]{BFM25}, where it is required that rectangles do not have overlapping sides, asking that $\FF^v (g_{n-1}\mathring{R}) \supset \FF^v(g_{n}R)$ rather than that 
$g_n(R)$ is above $g_{n-1}(R)$ (and analogously for $\FF^h$-saturations and belowness).  However, these conditions are in fact equivalent: if $g_n(R)$ is above $g_{n-1}(R)$ but they share a side for all, or infinitely many $n$, then this side contains a fixed point of $g_n g_{n-1}^{-1}$.  Hyperbolicity of fixed points now 
implies that either a slight enlargement $R'$ of $R$ will have the property that $\FF^v (g_{n-1}\mathring{R'}) \supset \FF^v(g_{n}R')$ for all $n$, or that the fixed point is singular. In the latter case, the conclusion follows because stabilizers of singular points are infinite cyclic (\cite[Lemma 4.13]{BFM25}).
\end{remark}

Finally, we turn to the closing property:
\begin{definition}[Closing property] \label{def:closing}
The action $G\curvearrowright \mc P$ has the \emph{closing property} if each point $x$ has two neighborhood bases $\{U_i\}$, $\{V_i\}$ with $V_i \subset U_i$
such that: 
\begin{itemize}
\item If $x$ is nonsingular and $g (V_i) \cap V_i  \neq \emptyset$, then $g$ fixes a point in $U_i$. 
\item If $x$ is singular, and $C_i$ is a connected component of $V_i \ssm (\FF^{h}(x) \cup \FF^{v} (x))$ with $g(C_i) \cap C_i \neq \emptyset$, then $g$ fixes either $x$, a point either in the component of $U_i \ssm (\FF^{h}(x) \cup \FF^{v} (x))$ containing $C_i$, or in an adjacent connected component.
\end{itemize}   
\end{definition}
This notion appears in \cite{barthelme2025pseudo}.  A slight variation of it was presented in \cite{BFM25} under the same name; however this one is more natural from the perspective of flows, and the results of \cite{BFM25} carry over unchanged with this replacement. 
Translated to the language of flows, the closing property says that a nearby first return of an orbit (nearby meaning to a same quadrant, in the singular case) implies the existence of a nearby periodic orbit, which is the essence of the closing lemma.

\smallskip

Our next goal is to show that the finite rectangle condition, uniform hyperbolicity, and closing, pass naturally to branched covers.  For this, we need the following subdivision lemma: 

\begin{lemma}[Subdividing rectangles] \label{lem_cut_rectangles_2}
Suppose that the action $G \curvearrowright \mc P$ has hyperbolic fixed points and the finite rectangle condition with respect to the collection $\mc R$. Let $O$ be any $G$--invariant, closed, discrete subset of $\mc P$, and 
let $\mc R'$ denote the collection of rectangles obtained by subdividing each rectangle $R \in \mc R$ along the horizontal segments through the points of $O \cap R$.  Then $\mc R'$ also satisfies the finite rectangle condition and has the property that each rectangle of $\mc R'$ meets $O$ only along its horizontal sides.
\end{lemma}

\begin{proof} 
Let $R$ be a rectangle of $\mc R$. Since $O$ is closed and discrete and $R$ is compact, $R \cap O$ is finite and so subdividing $R$ along the horizontal segments through points of $R \cap O$ produces a finite collection of subrectangles of $R$, each of which has its vertical sides along the vertical sides of $R$.  Let $\mc R'$ denote this collection.  

The fact that translates of rectangles in $\mc R'$ cover $\mc P$ follows immediately from the same fact for $\mc R$.  Thus, to show $\mc R'$ satisfies the finite rectangle condition, we need to check the coverings by rectangles above and below.  

To do this, fix $S$ in $\mathcal{R'}$ and let $R$ denote its parent rectangle in $\mathcal{R}$.  Consider the finite collection of rectangles $a_i R_i$ that are above $R$ and cover $R$.  For each $a_i R_i$, if $R_i$ contains points of $O$, then by $G$--invariance of $O$, their image under $a_i$ cannot be contained in the interior of $S$.  Thus, for each $i$, the intersection $S \cap a_i R_i$ lies in some subrectangle $a_i S_i$ where $S_i \in \mathcal{R'}$. 
We will show that either the $a_i S_i$ cover $S$ and are above $S$, or, up to replacing some or all of the $a_iS_i$ with a finite cover by translates of rectangles above $a_i S_i$, this new finite set covers $S$.  

To do this, first observe that since $a_i R_i$ is above $R$, either $a_i S_i$ is above $S$ (which is the desired situation),
or $a_i S_i$ has points of $O$ on its top \emph{and} bottom boundary (i.e. both horizontal sides),
each of which lies along the horizontal boundary of $S$. In this case, we say that the horizontal sides of $a_iS_i$ \emph{nest} in the horizontal sides of $S$, which implies that $a_iS_i \subset S$.
See \Cref{fig_chopping} for an illustration.

 \begin{figure}[h]
   \labellist 
     \small\hair 2pt
     \pinlabel $R_i$ at 400 70
     \pinlabel $S_i$ at 320 80
         \pinlabel $a_{i}S_i$ at 60 120
     \pinlabel $a_i$ at 200 20
      \pinlabel $R$ at -10 70
     \pinlabel $S$ at 120 120
\normalsize 
 \endlabellist
\includegraphics[width=9cm]{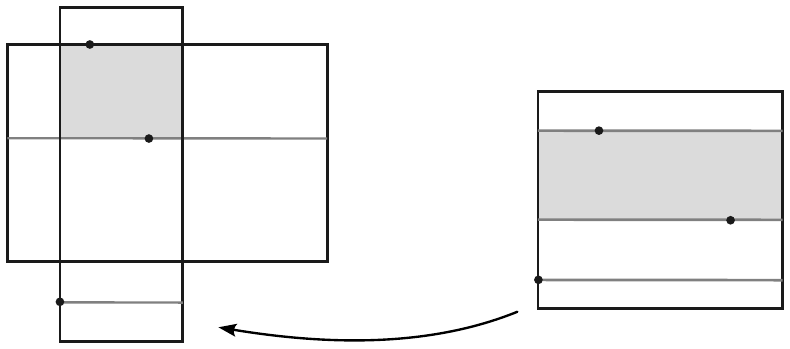}
\caption{If $a_{i}S_i$ is not above $S$, then $S$ is bounded above and below by segments containing points of $O$.}
 \label{fig_chopping} 
\end{figure}
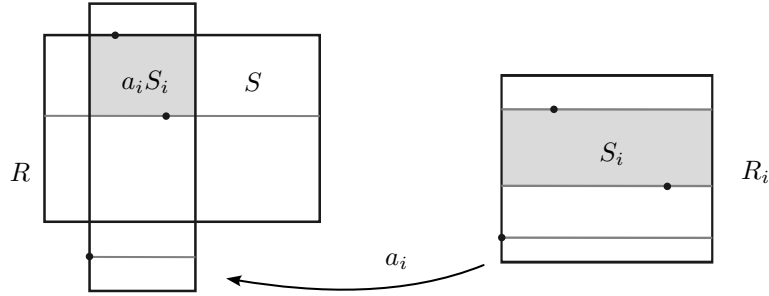

To treat this latter case, 
thinking of the rectangle $R_i$ as fixed, 
we consider the finite collection of rectangles $a_jR_j$ that cover and are above $R_i$. By repeating the previous argument to cover the rectangle $S_i$ and translating to $S$ using $a_i$, we either obtain a finite covering of $S$ by translates of the form $a_i S_i$ and $a_ia_j S_j$, each of which is above $S$, or for some $j$, $a_ia_j S_j$ again has points of $O$ on both its horizontal sides
and these sides lie along the horizontal sides of $a_iS_i$, and hence along $S$. In other words, the horizontal sides of $a_j S_j$ nest in the horizontal sides of $S_i$.

However, the total number of times that this process can continue is bounded by the number of rectangles in $\mc R'$. Indeed, otherwise we would encounter a single $s \in \mc R'$ that appears in distinct stages, meaning that there are $g_1,g_2 \in G$ such that the horizontal sides of $g_2 s$ nest in the horizontal sides of $g_1 s$ which in turn nest in the horizontal sides of $S$. But then the element $g_2g_1^{-1} \in G$ maps the rectangle $g_1 s$ to the rectangle $g_2 s \subset g_1s$ and this generates a nonhyperbolic fixed point of $g_2g_1^{-1} \in G$ along each of the horizontal sides of $g_1s$. This contradiction implies that the above process terminates in a finite collection of translates of $\mc R'$ that are above and cover $S$.

The proof for covering by rectangles below is immediate, since no cuts were made to the vertical direction of any rectangle.   
\end{proof}

\begin{remark}
The reader may verify that, if the set $O$ in \Cref{lem_cut_rectangles_2} is the set of singularities of $\mc P$, then the condition that the action has hyperbolic fixed points is in fact not needed and, moreover, that the number of iterations in the argument is at most two.
\end{remark}

\smallskip
We can now pass our conditions to branched covers:

\begin{proposition}  \label{prop_pass_to_cover}
Suppose $G$ acts on $(\mc P, \FF^h, \FF^v)$ with any of the following properties:
\begin{itemize} 
\item uniformly hyperbolic fixed points, 
\item the finite rectangle condition, or
\item transitivity.  
\end{itemize}
If $\wh {\mc P} \to {\mc P}$ is a $G$--compatible, regular branched cover and $\wh G$ is the group of all lifts of elements of $G$ to $\wh{{\mc P}}$, then the action of $\wh{G}$ on $\wh{{\mc P}}$ also has that same property.  
\end{proposition} 

In the proofs of \Cref{th:main}, we will only need to consider branched covers $\wh {\mc P} \to {\mc P}$ for which each branch value is a singularity of ${\mc P}$. This implies that each rectangle of ${\mc P}$ lifts to a rectangle of $\wh {\mc P}$. However, since we believe that more general branched covers could be useful in future work, we prove \Cref{prop_pass_to_cover} in its  general setting.  

\begin{proof} 
The fact that fixed points are hyperbolic obviously holds for a lifted action.  For uniformity, since the image of a rectangle in $\wh{{\mc P}}$ is a rectangle in ${\mc P}$, the failure of uniform hyperbolicity in $\wh{{\mc P}}$ pushes down to give failure in ${\mc P}$. 

If $\mc{R}$ is a finite collection of rectangles in ${\mc P}$ witnessing the finite rectangle condition, then 
we can first apply \Cref{lem_cut_rectangles_2} to replace $\mc R$ so that the branch values of $\wh {\mc P} \to {\mc P}$ occur only along the horizontal sides of rectangles in $\mc R$. Then
any choice of lift of each $R \in \mc{R}$ to $\wh{{\mc P}}$ forms a finite collection satisfying the finite rectangle condition.

Finally, if $G \cdot x$ is dense in ${\mc P}$, then $\wh G \cdot \wh{x}$ is dense in $\wh{{\mc P}}$ for any lift $\wh{x}$ of $x$.   Hence, transitivity also lifts, as required.
\end{proof} 

Because the closing property is framed differently around singular and nonsingular points, it respects branching over singular points:
 
\begin{proposition} \label{pass_to_cover2}
Suppose $G$ acts on $(\mc P, \FF^h, \FF^v)$ with the closing property and $\wh {\mc P} \to {\mc P}$ is a $G$--compatible, regular branched cover, branched over singular points.  Then the lifted action of $G$ on $\wh{{\mc P}}$ also has the closing property.  
\end{proposition} 

\begin{proof}
Whenever $p$ is a non-branch point, we can lift a pair of neighborhood bases, taken small enough to be disjoint from all branch points, witnessing the closing property in ${\mc P}$ to a pair in $\wh {\mc P}$.  
 
For the case of singular points, let $q$ be the pre-image in $\wh {\mc P}$ of a singular branch point $p$.   Let $V_p \subset U_p$ be a pair of neighborhoods as in the definition of the closing property for $p$.  We may without loss of generality assume that each $V_p$ and $U_p$ restricts to a rectangle in each quadrant.  Let $V_q \subset U_q$ be lifts of $V_p$ and $U_p$ to $\wh {\mc P}$. 

Suppose that for some connected component $R_q$ of $V_q \ssm (\FF^h(q) \cup \FF^v(q))$ and $\hat{g}$ lift of $g$ to $\wh G$, we have $R_q \cap \hat{g}R_q \neq \emptyset$.   Let $R_p$ be the projection of $R_q$ to $\mc P$, then $R_p \cap gR_p \neq \emptyset$ so $g$ either fixes $p$ (in which case $\hat{g}$ fixes $q$ and we are done) or has a fixed point in the quadrant of $U_p \ssm (\FF^h(p) \cup \FF^v(p))$ containing $R_p$, or an adjacent component.   The lift of $g$ which has a fixed point in the corresponding quadrant of $q$ (adjacent to or containing $R_q$) is then the unique lift moving $R_q$ to $\hat{g}R_q$, and thus agrees with $\hat{g}$. Thus $V_q \subset U_q$ satisfies the closing property for $q$. 
\end{proof}

\begin{remark}[Passing to finite index subgroups] \label{rem_finite_index}
Observe also that uniformly hyperbolic fixed points, the finite rectangle condition, the closing property, and transitivity all persist under passing to a finite index subgroup $G'$ in $G$ -- the only non-obvious step is in the finite rectangles condition, where one should enlarge $\mc{R}$ to a collection 
\[
\{ hR : R \in \mc{R}, h \text{ a right-coset representative of } G' \}.
\]  
\end{remark} 

\section{Flows with leafwise orientation} \label{sec_compactness} 
Recall our standing assumption that $(\mc P, \FF^h, \FF^v)$ has only prong singularities (at most one on any leaf) and no 
properly embedded $[0, \infty) \times [0,1]$.  Our goal in this section is to prove \Cref{thm_compactness_no_transitive}.
This theorem parallels the setting of  \cite{BFM25}, but with the added certification that the manifold is compact, which was missing from that work. 

To prove it, we first recall the framework of \cite{BFM25}, in which a (not necessarily compact) 3-manifold with an expansive flow is constructed out of an action on a bifoliated plane. 
In \Cref{subsec_compact} we show that under the hypotheses of \Cref{thm_compactness_no_transitive}, this manifold is compact in the nonsingular case, in \Cref{subsec_compact_singular} we treat the generalization to the singular case, and \Cref{subsec_reverse_implicaiton} shows the reverse implication, namely that orbit space actions have the uniformly hyperbolic fixed points, the closing property, and the finite rectangle condition.

\subsection{Construction of a 3-manifold} \label{subsec_3manifold}

Suppose that $G$ acts on $(\mc P, \FF^h, \FF^v)$ preserving orientation along leaves of $\FF^h$.  For $x \in \mc P$, denote by $\FF^h_>(x) \subset \FF^h(x)$ the set of points on the positive side of $x$. 
If $\mc P$ is nonsingular, we define 

 \[W_h := \{(x,t) \in \mc P \times \mc P \mid t \in \FF^h_>(x)\} \]
equipped with the subset topology from $\mc P \times \mc P$.

In the singular case, Lemma 4.13 of \cite{BFM25} shows that stabilizers of singularities are infinite cyclic.  Using this, for each $2k$ prong singularity $p \in {\mc P}$, let $r_1(p), \ldots, r_k(p)$ denote the rays in $\mc F^h_>(p)$  of the prong, and fix a proper homeomorphism $\sigma^p_j\colon r_1(p) \to r_j(p)$ that is equivariant with respect to $\mathrm{Stab}_G(p)$ in the following sense:  if $g \in \mathrm{Stab}_G(p)$ fixes all rays through $p$, then $\sigma^p_j(g(x)) = g \sigma^p_j(x)$.   One can define such homeomorphisms arbitrarily on a fundamental domain for the stabilizer of $r_1$ and then extend equivariantly; fixing $\sigma_1 = \mathrm{id}$. We then can extend the definition of $W_h$ above as follows.  

\begin{definition}
Define a space $W_h \subset {\mc P} \times 2^{\mc P}$ as follows. 
We say $(x, Y) \in W_h$ if: 
\begin{itemize}
\item $Y = \{y\}, \, y \in \FF^h_>(x)$, $x$ is not singular, and $\FF^h_>(x)$ has no singular point between $x$ and $y$, 
or
\item $Y = \{y, \sigma^p_2(y), \ldots, \sigma^p_k(y) \}$, $y \in \FF^h_>(x)$  and  $p \in \FF^h(x)$ is a 2k-prong, either between $x$ and $y$  or equal to $x$. 
\end{itemize} 
We topologize $W_h$ by saying that $(x_n, Y_n)$ converges to $(x, Y)$ if $x_n \to x$, and {\em some point} of $Y_n$ converges to some point of $Y$.  
\end{definition} 

Corollary 1.11 of \cite{BFM25} shows that, if $G$ has uniformly hyperbolic fixed points and the closing property,
 then the action of $G$ on $W_h$ is properly discontinuous and free, and $W_h/G$ is a 3-manifold $M_{\mc P}$ with expansive flow $\psi_{\mc P}$ whose orbit space is equivariantly isomorphic to $(\mc P, \FF^h,\FF^v)$.  By \cite{inaba1990nonsingular, paternain1993expansive},
 expansive flows on compact 3-manifolds are pseudo-Anosov, and thus our goal is to show that under the additional hypothesis of the finite rectangle condition $W_h/G$ is compact.  

For simplicity, we do this first in the case where $\mc P$ is nonsingular, and then explain the modification for the general case. 

\subsection{Compactness}  \label{subsec_compact}
We begin by describing the flow $\psi_{\mc P}$ from the construction of \cite{BFM25}.  First, observe that for every point $x \in {\mc P}$, the space of points in $W_h$ with first coordinate $x$ is homeomorphic to a line (this is also true in the singular case).  This line corresponds to an orbit of the lift $\widetilde{\psi}_{\mc P}$ of $\psi_{\mc P}$ to $W_h \cong \mathbb{R}^3$, the universal cover of $M_{\mc P}$.  
We orient $W_h$ and the orbits of $\widetilde{\psi}_{\mc P}$ so that orbits are vertical lines in $\mathbb{R}^3$ and, as $t$ moves towards $x$, (or as $Y$ moves towards $x$ in the singular case), we travel vertically up the flowline $(x,t)$, or $(x, Y)$. This is consistent with the conventions established following \Cref{def:above}.

In the nonsingular case, every point of $W_h$ has a neighborhood of the following form:
A \emph{closed good rectangle neighborhood} (or closed good neighborhood for short) is 
a subset $V = (R \times R') \cap W_h$ such that:
\begin{enumerate}[label=(\roman*)]
\item $R$ and $R'$ are closed rectangles 
\item The saturations of $R$ and $R'$ by $\FF^h$ leaves agree, and 
\item $R \cap R' = \emptyset$.
\end{enumerate}
An open good rectangle neighborhood is defined in the same way, where $R$ and $R'$ are open rectangles.  

\begin{proposition} \label{prop_nonsingular}
Assume ${\mc P}$ is nonsingular and satisfies the hypotheses of \Cref{thm_compactness_no_transitive}. Then there is a finite collection of closed good rectangle neighborhoods whose $G$-translates cover $W_h$.  Consequently, $W_h/G$ is compact. 
\end{proposition}

\begin{proof}
Let $\mc R = \{R_1, \ldots R_n \}$ be a family of rectangles witnessing the finite rectangle condition of the action.  Let $\mc A$ and $\mc B$ denote the finite sets of elements of $G$ taking rectangles above (resp. below) each other as specified in the condition. 
For each $R_j$, choose $R_j'$ such that $V_j = (R_j \times R_j') \cap W_h$ is a good neighborhood.  Such an $R_j'$ exists because a compact rectangle can always be included in a slightly larger rectangle. 

If $G (\bigcup_j  V_j)$ already covers $W_h$, we are done, otherwise we will iteratively add finitely many more good neighborhoods to the collection. For this, we will use the following consequence of uniform hyperbolicity. 
\begin{claim} \label{claim:good_position}
Let $R_j \in \mc R$.  For each point $p \in R$, and $t \in \FF^h_>(p)$ there exists $R_i \in \mc R$ and a translate $g R_i$ containing $p$, where $g$ is a word in $\mc A$, such that $g R'_i$ is on the negative side of $t$.  In particular, we can find such a $g R_i'$ on the negative side of $R_j'$. 
\end{claim}

\begin{proof}[Proof of claim]
Let $R_j$ be given and $p \in R_j$.  To ease notation, we notate the pair $R_j, R_j'$ by $R$ and $R'$ instead. 
For each good neighborhood $V_i$, let $S_i$ denote the minimal rectangle containing $R_i$ and $R_i'$.  
 By iterative applications of the finite rectangle criterion, we can find a tower $w_{k}R_{i_k}$ where  $w_{k}R_{i_k}$ is above $w_{k-1}R_{i_{k-1}}$ and all contain $p$.  Each $w_k$ is a word in $\mc A$.  By pigeonhole, some fixed rectangle appears infinitely often in this sequence.  Let $R_i$ be such a rectangle, and let $\wh R = w R_i$ be its first appearance in the tower. 
Thus, there are words $g_n$ in the $w_k$ such that $g_n \wh R$ is above $g_{n_1} \wh R$.  Since fixed points are hyperbolic, this implies also that $g_n( wS_i)$ is above $g_{n-1}(w S_i)$.  By uniform hyperbolicity, the $\FF^v$ saturation of the $g_n( wS_i)$ limits to a single leaf (necessarily containing $p$). Since $g_n(w R'_i) \subset g_n(w S_i)$, we are done. 
\end{proof}

Fix $R \in \mc R$.   Using the claim, for each $p \in R$ we have a rectangle $a(p) R_{i(p)}$ above $R$, containing $p$ and such that 
 $a(p) R'_{i(p)}$ is on the negative side of $R'$.  We next argue how to reduce this to a finite collection:  First, the vertical boundaries of $R$ can each be covered by a single rectangle of the form $a(p) R_{i(p)}$, fix these.  Take compact $S \subset R$ so that the interior of $S$ and the interior of the chosen boundary rectangles still covers $R$.  Consider the open cover of $S$ given by the two fixed boundary rectangles, and all sets of the form 
 \[
 \mathrm{int} \left ( a(p_1) R_{i(p_1)} \cup a(p_2) R_{i(p_2)} \right ),
 \]
 where $p_1, p_2 \in S$, and $\mathrm{int}$ denotes relative interior (in $S$). This allows us to consider adjacent rectangles which meet only along their boundary, and merge them into a single open set.  By compactness of $S$, finitely many such sets cover, which means that there is a finite collection of $a(p) R_{i(p)}$ which cover $R$.   
 
 \smallskip
Fix this finite collection, and call it $\mathcal{C}(R)$.  
For rectangle $a_i R_i$ in $\mathcal{C}(R)$, we add a good neighborhood $U = (A \times A') \cap W_h$ where $A = R_1 \cap a_{i}R_{i}$ and $A'_{i}$ is the region of $\FF^h(R)$ bounded between $a_{i}R'_{i}$ and $R'$.  See 
 \Cref{fig_new_rectangles}.  
These are the finitely many new closed good neighborhoods associated to $R$, and we repeat this procedure for all $R \in \mc R$. Let $\mc U$ denote the collection of all $V_j$ and all these new closed good neighborhoods.   

 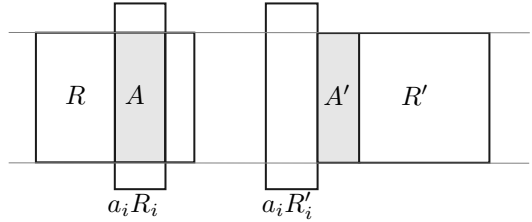
\begin{figure}[h]
   \labellist 
     \small\hair 2pt
     \pinlabel $A$ at 92 70
     \pinlabel $A'$ at 235 70
    \pinlabel $R$ at 50 70
     \pinlabel $R'$ at 290 70
    \pinlabel $a_{i}R_i$ at 91 -8
     \pinlabel $a_{i}R_i'$ at 200 -8
 \endlabellist
\includegraphics[width=7cm]{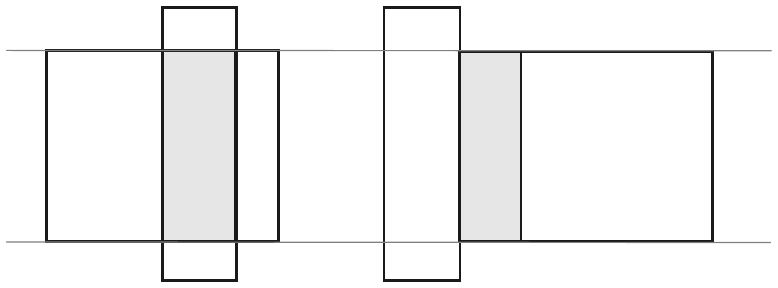}
\caption{New rectangles $A$ and $A'$ for the proof of \Cref{prop_nonsingular} }
 \label{fig_new_rectangles} 
\end{figure}

We claim now that, for each $(p_0, t_0)$ in any $V = V_j$, the entire ray $\{ (p_0, t) : t_0 \in \FF^h_>(t) \}$ corresponding to the ``upwards" orbit of $(p_0, t_0)$ under the flow $\psi_{\mc P}$, is covered by $G$-translates of elements of $\mc U$. 

To see this, we argue as follows.  Let $(p_0, t_0) \in V = (R \times R') \cap W_h$ be given.  As argued above, there is some $R_0 \in \mc R$ and $a_0 \in G$ such that $a_0 R'_0$ is on the negative side of $R'$, and in the previous step we created also a neighborhood of the form $A_0 \times A'_0 \cap W_h$ such that $A'_0$ contains the segment of $\FF^h(p)$ between $a_0 R'_0$ and $R'$. Thus, every point $(p_0, t)$ where $t$ lies either in $a_0 R'_0$ or between $a_0 R'_0$ and $t_0$, is covered by a translate of one of the good neighborhoods in $\mc U$. 

Now consider $a_0^{-1}(p_0)$ which lies in $R_0$. Repeating the argument above with $R_0$ playing the role of $R$, we can find $R_1 \in \mc R$ and $a_1$ such that $a_1 R'_1$ is on the negative side of $R_0'$ and every point of the form 
$(a_0^{-1}p_0, t)$ where $t$ lies either in $a_1 R'_1$ or between $a_1 R'_1$ and $a_0^{-1}t_0$, is covered by a translate of one of our finite set of good neighborhoods.  Applying $a_0$, we have now that every point of the form $(p_0, t)$ where $t$ is between the negative (in the orientation) boundary of $a_0 a_1 R_1'$ and $t_0$ is now covered.  Repeat this procedure inductively. 

As in \Cref{claim:good_position}, uniform hyperbolicity implies that the $\FF^v$ saturation of the iterates $a_0 a_1.... a_n (R_n')$ approaches the $\FF^v$ leaf of $p$.  This shows that the ray $\{ (p_0, t) : t_0 \in \FF^h_>(t) \}$ is contained in the union of translates of our good neighborhoods.  
\smallskip

By using a similar procedure, considering translates below the elements of $\mc R$ and adding finitely many additional closed good neighborhoods, we will next ensure that the downwards orbit of any point in any $V_i$ is covered by $G$-translates of a finite collection of closed good neighborhoods.
For this, we use the following analog of \Cref{claim:good_position}. 

\begin{claim} \label{claim:good_position2}
Let $R \in \mc R$.  For each point $p \in R$, and $t \in \FF^h_>(p)$ there exists $R_i \in \mc R$ and a translate $g R_i$ containing $p$, where $g$ is a word in $\mc B$. 
such that $g R'_i$ is on the positive side of $t$.  In particular, we can find such a $g R_i'$ entirely contained on the positive side of $R'$. 
\end{claim}

\begin{proof}[Proof of claim]
As in the proof of \Cref{claim:good_position}, the finite rectangle condition and pigeonhole principle can be used to find a single rectangle $R_i \in \mc R$ and a sequence of translates $g_n R_i$ containing $p$, where $g_n R_i$ is below $g_{n-1} R_i$.  In particular the elements $g_n$ are all distinct.  Since the action of $G$ on $W_h$ is properly discontinuous, the positive sides of the $g_n R_i$ must escape all compact sets. Since $g_n R'_i$ is on the positive side of $g_n R_i$, this proves the claim.  
\end{proof} 

With this claim in hand, we can imitate the construction above, adding finitely many good rectangle neighborhoods to cover the ``gaps" between successive layers of translates of the existing neighborhoods.  In detail, 
using \Cref{claim:good_position2}, for each $R \in \mc R$, we can find a finite collection of translates $b_{i} R_{i}$ which are all below $R$, cover $R$, and such that $b_{i} R'_{i}$ is on the positive side of $R'$.  
For each element of this finite collection, we add a good neighborhood of the form $U' = (B_{i} \times B'_{i}) \cap W_h$ where $B_{i} = R \cap b_{i}R_{i}$ and $B'_{i}$ is the region of $\FF^h(b_{i_k}R_{i_k})$ bounded between $R'$ and $b_{i}R'_{i}$.  
We repeat this procedure for all $R \in \mc R$, and let $\mc U'$ denote the collection of all $U'$ and all these new closed good neighborhoods. 

The argument to show that, for each $(p_0,t_0)$ in a good neighborhood, the ray $\{ (p_0, t) : t \in \FF^h_>(t_0) \}$ corresponding to its downwards orbit is contained in the union of the translates of the elements of $\mc U'$ is exactly the same as the previous case. 

Thus, we can cover the {\em full orbit} of each point in a good neighborhood $V_i$ by $G$-translates of a collection of finitely many compact sets.   
Since by hypothesis the translates of $R_i$ cover the orbit space ${\mc P}$, we conclude that $W_h$ is covered by translates of finitely many compact sets, as desired.  
\end{proof} 

\subsection{Modifications when $\mc P$ is singular}  \label{subsec_compact_singular} 
In this section we describe the necessary modifications to the proof of \Cref{prop_nonsingular} to treat the singular case, proving the following: 

\begin{proposition} \label{prop_singular_case}
Suppose $\mc P$ satisfies the hypotheses of \Cref{thm_compactness_no_transitive}. Then there is a finite collection of closed good rectangle neighborhoods whose $G$-translates cover $W_h$.  Consequently, $W_h/G$ is compact. 
\end{proposition} 

The proof follows exactly the same outline, but more care is needed because good neighborhoods in $W_h$ look different than in the nonsingular case whenever they contain a prong singularity.

\begin{proof}[Proof of \Cref{prop_singular_case}]
We follow the strategy of \Cref{prop_nonsingular}.   
Using the set of prong singularities as the set $O$ in \Cref{lem_cut_rectangles_2},
 we may assume that we have a collection $\mathcal{R}$ of rectangles witnessing the finite rectangle condition and so that all prong singularities are on horizontal sides.   

With this modification, rectangles in $\mathcal{R}$ can be extended to closed good rectangle neighborhoods as follows: For each $R_i$ which has no singularities on its sides, choose a companion rectangle $R'_i$ so that $V_i = (R_i \times R_i') \cap W_h$ is a closed good neighborhood.  See the left side of \Cref{fig_prong_boxes}.
For any rectangle $R_i$ that has a singularity on top or bottom, we take an open rectangle $\mathring{R}_i$ such that 
$\mathring{V}_i := (\mathring{R}_i \times \mathring{R}_i') \cap W_h$ is an {\em open} good rectangle neighborhood, and let $V_i$ denote its closure in $W_h$.   One should think of $V_i$ as consisting of pairs of points $(x,t)$ in $\mc P$ where $x \in R_i$, and $t \in R'_i$ is on the same leaf as $x$, {\em except} for along the singular boundary leaves, where $t$ is replaced instead with a finite set.  An illustration is given in the right side of \Cref{fig_prong_boxes}; here the finite sets are pairs of 2 points each, on the segments of the prong leaf indicated in blue.   In practice, we will simply parameterize these sets of 2 (or $k/2$, for a $k$-prong singularity), by using the uniqe element in  $R'_i$ as a representative. 

 \begin{figure}[h]
\includegraphics[width=13cm]{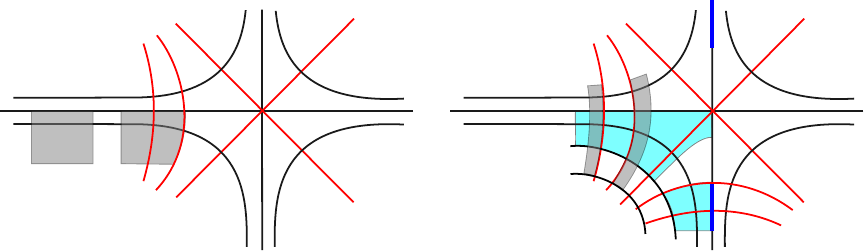}
\caption{Illustration of good rectangle neighborhoods}
 \label{fig_prong_boxes} 
\end{figure}

Now we run the same argument as in \Cref{prop_nonsingular}.   \Cref{claim:good_position} applies verbatim. (Note that when $a_i R_i$ is above a rectangle $R$ with singularity, then $a_i R_i'$ cannot be separated from $a_i R_i$ by the singular leaf.)
Thus, we assume our finite collection are in the form given by the Claim, and use it to add finitely more closed good neighborhoods, each of which has interior of the form: 
 $\mathring{U} = (\mathring{A} \times \mathring{A}') \cap W_h$ where $A = R_1 \cap a_{i}R_{i}$ and $A'_{i}$ is the region of $\FF^h(R)$ bounded between $a_{i}R'_{i}$ and $R'$.   Again, we always take closure in $W_h$.  These, by construction, also have the property that prong singularities occur only on their horizontal boundary sides, not on the vertical sides, so we can treat them in the same way.   
 
One then argues exactly as before to show that the forward orbit of any point in a translate of these finitely many closed good neighborhoods is covered by translates of these good neighborhoods.  The argument for backwards orbits applies verbatim.  
\end{proof} 

\subsection{Proof of reverse implication}  \label{subsec_reverse_implicaiton} 
In this section we prove the following:  
\begin{proposition} \label{prop:only_if}
Let $\phi$ be a pseudo-Anosov flow on a compact manifold.  Then $\phi$ satisfies the finite rectangle condition, the closing property, and has uniformly hyperbolic fixed points.  
\end{proposition}

\begin{proof} 
The closing property and uniformly hyperbolic fixed points are standard, and explained in \cite{BFM25}.  For the finite rectangle condition, we will use Markov partitions for the flow.  

 A {\em Markov partition} for $\phi$ is a finite family of pairwise disjoint transverse
rectangles $R_1, \ldots, R_m$ such that, there is a uniform bound $T$ such that for any $x \in M$, there is $t \in [0,T]$ such that $\phi^t(x) \in \bigcup_i R_i$ and
the (forward) first return map $f: \bigcup_i R_i \to \bigcup_i R_i$ is defined and satisfies the following 
\begin{enumerate} 
\item For all $i, j$, if $f(\mathring{R}_i) \cap \mathring{R}_j \neq \emptyset$, then the closure of every connected component of $f(\mathring{R}_i) \cap \mathring{R}_j$ is a vertical subrectangle of $R_j$.  
\item For all $i, j$, if $f^{-1}(\mathring{R}_i) \cap \mathring{R}_j \neq \emptyset$, then the closure of every connected component of $f(\mathring{R}_i) \cap \mathring{R}_j$ is a horizontal subrectangle of $R_j$.
\end{enumerate} 

As shown by Ratner \cite{ratner1973markov} and Brunella \cite{brunella1995surfaces} in the transitive case and extended by Iakovoglou \cite{iakovoglou2025markovian}
to the general case, any pseudo-Anosov flow admits a Markov partition.  In fact, one can take the sides of the rectangles to contain any prescribed finite collection of periodic orbits.  

Given such a Markov partition, we claim that the images in the orbit space of each rectangle provide the desired family.  That their orbits cover ${\mc P}$ is given by the first condition.  To see the covering condition, we use the first return map.
First, note that the condition that each point meets a rectangle in finite time implies that the first return map is surjective.  

Now, fix some lift $\wt R_i$ of each $R_i$ to $\wt M$, and let $R_j^{\mc P}$ denote its image in the orbit space ${\mc P}$.  Since we may choose the sides of $R_i$ to be periodic, following a periodic orbits along a vertical side eventually returns to $R_i$. Thus, there is a power of the first return map taking $R_i$ to intersect itself, necessarily the closure of the intersections of interiors of $R_i$ and its image is a vertical subrectangle.  On the level of $\wt M$, this means that by flowing forward along orbits starting at $\wt R_i$, one eventually meets another translate $g \wt R_i$ of $R_i$, overlapping as a vertical subrectanlge.  On the level of the orbit space, this implies that $g R_i^{\mc P}$ is {\em below }$ R_i^{\mc P}$, and $g^{-1} R_i^{\mc P}$ is above $R_i^{\mc P}$, covering the vertical side containing the periodic orbit.  One can use another translate to cover the other vertical side similarly, and then (using a compactness argument) cover the interior also by finitely many images of rectangles in $\bigcup_j R_j$ under the first return map, which correspond as above to translates of the $R_j^{\mc P}$ under elements of $\pi_1(M)$.  

Flowing backwards, or using $f^{-1}$ instead of $f$, one similarly finds a finite cover of each $R_i^{\mc P}$ by translates of rectangles which are below. 
\end{proof}

\section{Branched covers, veering triangulations, and a dual flow}
\label{sec:veer}
Let $\mc P$ be a bifoliated plane, which we now allow to have odd prongs.  In this section we fix an action of a group $G$ on $\mc P$ not assumed to preserve a leafwise orientation (such an orientation may not even exist), but satisfying the hypotheses of \Cref{th:main}.  From this, we show how to construct a
manifold $M$ with fundamental group $G$ and a pseudo-Anosov flow $\varphi$.  In \Cref{sec:same_action}, we show the action of $G$ on $\orb_\phi$ is conjugate to the original action of $G$ on ${\mc P}$.  

Let $O$ denote the (possibly empty) set of odd prongs in $\mc P$.  To ``unwind" these to even prongs, we do the following. 

\begin{definition} [Orientation branched cover]
Let $\pi \colon  \mc P_\ast \to \mc P$ be the universal orbifold cover of $\mc P$ equipped with an orbifold structure where each point of $O$ is an order 2 cone point.  Let $H$ denote the covering group of $\mc P_\ast \to \mc P$, and $\FF^{v/h}_\ast$ the lifts of the foliations $\FF^{v/h}$ to $\mc P_\ast$.  
\end{definition}

Note that $\pi \colon  \mc P_\ast \to  \mc P$ is a $G$--compatible normal branched cover of bifoliated planes, as in \Cref{sec:branched}.

\begin{remark}
The plane $\mc P_\ast$ can also be described as a two-step cover in terms of the foliated structure of $\mc P$: 
First, pass to the double branched cover $D$, branched over $O$, defined by the orientation homomorphism on $\FF^v$.  Then $\mc P_\ast$ is the universal covering of $D$, so $\pi$ factors through $D$.   
\end{remark} 

We now recall the assumptions on $G$ 
\begin{convention}
For the remainder of this paper, we suppose that $G$ is a torsion-free group acting (topologically) transitively, by foliation-preserving automorphisms of $\mc P$ and preserving an orientation. 
Assume that $G$ has the closing property, uniform hyperbolicity and the finite rectangle condition.  
\end{convention} 

Denote by $G_\ast$ the group of all lifts of $G$ to $\mc P_\ast$.  This sits in an exact sequence  
\[
1 \to  H  \to G_* \to G \to 1.
\]
where $H$ is the covering group of $\mc P \to \mc P_\ast$ as above.  
Let $G_\ast^+$ be the index 2 subgroup of $G_\ast$ preserving orientation 
along leaves of ${\mc P}_\ast$.  

\begin{lemma} \label{lem:torsion_free}
With the setup as above, if $G$ is torsion-free, then so is $G^+_*$.
\end{lemma}

\begin{proof}
We will show that any torsion element of $G_*$ reverses the orientation along leaves of $\mc P_*$ and hence is not contained in $G^+_*$. To this end, let $g \in G_*$ have finite order. Then because $G$ is torsion-free, $g \in H$. Since $H$ is the covering group of the universal orbifold cover $\pi \colon \mc P_* \to \mc P$, $g$ is an order two element generating the stabilizer of some $p \in \pi^{-1}(O)$. This implies that $g$ reverses the orientation on the leaves of $\mc P_*$ because the leaves of $\mc P$ are not locally orientable at any point of $O$. Hence, $g \notin G^+_*$ and the proof is complete.
\end{proof}

Using this and \Cref{thm_compactness_no_transitive} we have: 
\begin{corollary}  \label{cor_N_psi}
There is an orientable, closed manifold $N$ with $\pi_1(N) =G_\ast^+$ and a transitive pseudo-Anosov flow $\psi$ on $N$ so that $G_\ast^+ \curvearrowright {\mc P}_\ast$ agrees (up to conjugacy) with the orbit space action $G_*^+ \curvearrowright \orb_\psi$. 
\end{corollary} 

\begin{proof} 
 By \Cref{prop_pass_to_cover}, \Cref{pass_to_cover2}, and \Cref{rem_finite_index}, the action of $G_\ast^+$ is transitive, and has uniformly hyperbolic fixed points, the closing property, and the finite rectangle condition.  By \Cref{lem:torsion_free}, $G^+_\ast$ is torsion free.  Thus, \Cref{thm_compactness_no_transitive} gives the desired conclusion.  
 \end{proof}

\subsection{Orbits killing perfect fits}
\label{sec:pfits}
Our first goal is to use the flow $\psi$ from \Cref{cor_N_psi} to give a veering structure for the action of $G$ on ${\mc P}$ (after removing a suitable collection of orbits).  To do this, we recall the basic set-up of constructing a veering triangulation from an orbit space.  

A {\em perfect fit rectangle} in a bifoliated plane is the image of a properly embedded copy of $[0,1]^2 \ssm \{(0,0)\}$ via a foliation-respecting homeomorphism.  
For a transitive pseudo-Anosov flow $\psi$ on a closed orientable 
$3$-manifold $N$, Fried \cite{fried1983transitive} and Brunella \cite{brunella1995surfaces} show that there is a finite collection of closed orbits of $\psi$ that kill all perfect fits. (This terminology is taken from \cite{landry2025transverse}). This means that after this collection of orbits is lifted to $\wt N$ and projected to $\orb_\psi$,  the resulting $\pi_1(N)$--invariant, closed, discrete subset of $\orb_\psi$ intersects each perfect fit rectangle.
See Tsang \cite[Proposition 2.7]{tsang2022constructing} for a proof (using the language of perfect fits) of the stronger result that this can be achieved with a single orbit.

In our setting, this means that there is a closed, discrete set of points $\kappa_\ast \subset \mc P_\ast$, each fixed by some nontrivial element of $G^+_\ast$ and consisting of finitely many  $G^+_\ast$ orbits, such that $\kappa_\ast$ meets each perfect fit rectangle in ${\mc P}_\ast$. In this case, we also say that 
 ${\mc P}_\ast$ has \emph{no perfect fits relative to $\kappa_\ast$}. 

The next proposition shows how to achieve this property in $\mc P$.

\begin{proposition}\label{prop:npfs}
There is a closed, discrete set of
points $\kappa \subset \mc P$, consisting of finitely many $G$--orbits of points that are fixed by some nontrivial element of $G$, such that ${\mc P}$ has no perfect fits (or infinite strips) relative to $\kappa$. 

Moreover, the stabilizer of each point of $\kappa$ is infinite cyclic. 
\end{proposition}

\begin{proof}
Let $\kappa'$ be the image of $\kappa_\ast$ under the branched covering map $\pi \colon  \mc P_\ast \to  \mc P$.
Observe that this consists of finitely many $G$ orbits, and has no perfect fit rectangles relative to $\kappa'$, since any perfect fit rectangle disjoint from $\kappa'$ would lift to a perfect fit rectangle in $\mc P_\ast$, disjoint from $\kappa_\ast$.  
Since $N$ is compact, there are also only finitely many orbits of prong singularities under $G^+_\ast$, which form (by definition) a closed, discrete set.  Let $\kappa$ denote the union of $\kappa'$ with the projection of these singularities; again this consists of finitely many $G$-orbits in $\mc P$.  

It remains to show that each $p \in \kappa$ has infinite cyclic stabilizer. 
If $p_* \in \mc P_*$ with $\pi(p_*) = p$, then its $G_*$ stabilizer $G_*(p_*)$ fits into the sequence:
\[
1 \to H(p_*) \to G_*(p_*) \to G(p) \to 1.
\]
The index $2$ subgroup $G^+_*(p_*) \le G_*(p_*)$ is infinite cyclic because it is nontrivial by assumption and nontrivial stabilizers of $G^+_* \curvearrowright \mc P_*$ are infinite cyclic because this action is conjugate to $G^+_* \curvearrowright \orb_\psi$ (\Cref{cor_N_psi}). Hence, $G_*(p_*)$ is virtually cyclic, and because $H(p_*)$ is finite, $G(p)$ is also virtually cyclic. Because $G$ is torsion-free (\Cref{lem:torsion_free}), $G(p)$ is infinite cyclic and so the proof is complete.
\end{proof}

\subsection{The veering triangulation} 
\label{sec:veering}
A {\em veering triangulation} is an ideal triangulation of a cusped 3-manifold, built from combinatorial data associated to certain group actions on bifoliated planes. Rather than directly use their general combinatorial definition, we will instead rely on the construction introduced by Agol--Gu\'eritaud and recorded in \cite{LMT21, schleimer2024loom}.

We will specifically apply the Agol--Gu\'eritaud construction 
of a veering triangulation to the action of $G$ on $\mc P$ relative to $\kappa$.  Since our plane $\mc P$ is not yet known to be the orbit space of a pseudo-Anosov flow (as in \cite[Section 4]{LMT21} or \cite[Chapter 2]{tsang2023veering}) and is not itself a \emph{loom space} as in \cite{schleimer2024loom},
we cannot quote results verbatim to show this construction applies.  So, we instead give a sketch of the construction to show that it works in our setting. The reader will find the story similar to the presentation in the references above and can consult them for additional details.   Throughout, we will still refer to this as the \emph{Agol--Gu\'eritaud construction}.

\begin{remark}
As an alternative approach, it is possible to show that the universal cover of $\mc P \ssm \kappa$ \emph{is} a loom space, so one could produce there an infinite cover of the desired veering triangulation. However, since we are interested in constructing the manifold whose fundamental group is $G$ itself, we have chosen to proceed directly from the action of $G$ on $\mc P$. 
\end{remark}

A (closed) rectangle in $\mc P$ is \emph{maximal} (or a \emph{tetrahedra rectangle}) relative to $\kappa$ if it is maximal
among all rectangles whose interiors are disjoint from $\kappa$.  Since $\kappa$ is fixed, we typically just call these maximal rectangles. Note that a rectangle is maximal if and only if it contains a point of $\kappa$ (possibly a singularity) in the interior of each of its sides. 

\begin{figure}[htbp]
\begin{center}
\includegraphics{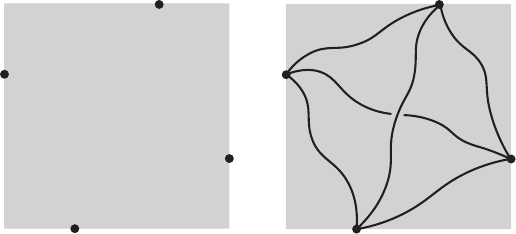}
\caption{A maximal rectangle $R$ and its tetrahedron $t_R$. Note that its bottom edges is drawn below its top edge. Figure from \cite{LMT21}.}
\label{LMT_tet}
\end{center}
\end{figure}

A {\em ideal tetrahedron} is a (simplicial) tetrahedron with its vertices removed.  
For each maximal rectangle $R$ in $\mc P$ (relative to $\kappa$), we associate an ideal tetrahedron $t_R$, whose ideal vertices are points of $\kappa$. See \Cref{LMT_tet}. This tetrahedron has a \emph{bottom edge} joining the points of $\kappa$ along the vertical sides of $R$ and a \emph{top edge} joining the points of $\kappa$ along the horizontal sides of $R$. We coorient the faces of $t_R$ so that the faces containing the bottom edge (called the \emph{bottom faces}) point into $t_R$ and the faces containing the top edge (the \emph{top faces}) point out of $t_R$.

Define
\[
\wt V := \bigcup_{R} t_R / \sim
\]
where the union is over all maximal rectangles and $\sim$ denotes the following face pairing of the tetrahedra:
a face of tetrahedron $t_R$ is paired with a face of a tetrahedron $t_{R'}$ if and only if the vertices of each face correspond to the same three points of $\kappa$ in $\mc P$. The ends of $\wt V$, which are called \emph{cusps}, are in natural bijection with the points of $\kappa$. We will soon see that $\wt V$ is an ideal triangulation of a $3$-manifold and so we refer to it informally in the meantime as an {\em ideal triangulation}. This fact can also be verified directly from the local structure of maximal rectangles in $\mc P$.

\smallskip
Since the action of $G$ preserves $\kappa$, it also preserves the set of maximal rectangles as well as the pairing relation. Hence, there is a well-defined simplicial action $G\curvearrowright \wt V$. Using the hyperbolic dynamics, it is easy to see that this action preserves no simplices (recall vertices have been removed) and so the quotient $\wt V /G$ is an ideal triangulation.  Let $V$ denote $\wt V /G$.

\begin{proposition}
\label{prop:its_veering}
The ideal triangulation $V$ is a veering triangulation 
 of an oriented $3$--manifold with torus cusps and finitely many tetrahedra.
\end{proposition}

The coorientations on the faces of $V$ determine a natural smoothing of the $2$-skeleton of $V$ making it into a branched surface, where the branch locus is exactly the set of edges of $V$. See the left-hand side of \Cref{fig:tet_ladder}. We will not use this branched surface structure directly, but we will refer to the induced train track structure on the cusps of $V$ in \Cref{sec:slopes}.

\begin{proof}
This can be verified directly, following exactly the same proof as in \cite[Lemma 4.5]{LMT21}, by using the condition that $\mc P$ has no perfect fits (or product regions) disjoint from $\kappa$ together with the facts that $\kappa$ contains finitely many $G$-orbits and each point in $\kappa$ has infinite cyclic stabilizer (see \Cref{prop:npfs}). 

Alternatively, here is a shortcut in our setting. By construction, $O \subset \kappa$ and so the branched cover $\pi \colon \mc P_\ast \to \mc P$ restricts to an honest covering space $\mc P_\ast \ssm \kappa_\ast \to  \mc P \ssm \kappa$. 
The associated cover $\wt V_d \to \wt V$ can be obtained by applying the Agol--Gu\'eritaud construction directly to $ \mc P_\ast$ (relative to $\kappa_\ast$). However, the action $G^+_d \curvearrowright \mc P_\ast$ is already known to be conjugate to the action on the orbit space of a transitive pseudo-Anosov flow on a compact manifold by \Cref{cor_N_psi}. This is the setting of the original Agol--Gu\'eritaud construction (as exposited in \cite{LMT21}), and so we may apply  \cite[Theorem 4.7]{LMT21} to see that $\wt V_\ast / G^+_\ast$ is a veering triangulation of an oriented $3$--manifold with torus cusps and finitely many tetrahedra. Since $\wt V / G = \wt V_\ast /G_\ast$, this together with the fact that $G_\ast/G_\ast^+$ acts on $\wt V_\ast / G^+_\ast$ preserving its orientation,
establishes the same for $V = \wt V / G$. 
\end{proof}

We refer the reader to \cite[Section 4]{LMT21} for an explanation of how the `veering structure' arises from the above construction.
We remark that $\wt V_\ast / G^+_\ast$ in the preceding proof is what Parlak calls the `edge-orientation double cover' of $V$ \cite{parlak2021taut}.

\subsection{Dehn filling on cusps}
\label{sec:slopes}
Our next goal is to prove the following: 
\begin{proposition}
\label{prop:veering_th}
For each torus cusp $T$ of $V$, there is a slope $s_T$ such that 
the closed manifold $M$ obtained by Dehn filling $V$ along the slopes $\{s_T\}$ supports a transitive pseudo-Anosov flow $\varphi$ which has a collection of closed orbits $\wh \kappa_\phi$ of $\varphi$ (containing all singular orbits) so that: 
\begin{enumerate}
\item there is an embedding $V \subset M$ transverse to $\varphi$ such that $M \ssm \wh \kappa_\varphi = V$,
\item  this
lifts to an embedding $\wt V \subset \wt M$ that is equivariant with respect to a natural isomorphism $G = \pi_1(M)$. Here, $\wt M$ is the universal cover of $M$.
\item $V$ can be recovered by applying the Agol--Gu\'eritaud construction to the orbit space action $\pi_1(M) \curvearrowright \orb_\varphi$ relative to $\kappa_\phi$. Here $\kappa_\phi$ is the projection to $\orb_\varphi$ of the preimage of $\wh \kappa_\phi$ in $\wt M$.
\end{enumerate}
\end{proposition}

Choosing the slopes $\{s_T\}$ so that the filled manifold has $\pi_1(M) \cong G$ is not too difficult, 
the point that requires significant care is to produce a transverse flow $\varphi$ which is pseudo-Anosov, and from which we can then prove the other properties listed above.  We will use the following result of Agol--Tsang (but see \Cref{rem_correspondence}): 

\begin{theorem}[Theorem 5.1 of \cite{AgolTsang}] \label{thm_AT5.1}
Suppose $\mr M$ admits a veering triangulation and let $\mc L_T \in H_1(T)$ 
denote the ladderpole class on each torus cusp $T$.
Then the filling $M$ of $\mr M$ along the slopes $(s_T)$ admits a transitive pseudo-Anosov flow if 
$| \langle [s_T], \mc L_T \rangle | \geq 2$ 
for all $T$.

Furthermore, there are closed orbits $\gamma_T$ isotopic to the cores of the filling solid tori, such
that each $\gamma_T$ is a prong with 
$\langle [s_T], \mc L_T \rangle$-many stable rays,
and $\varphi$ has no perfect fits relative to $\{\gamma_T\}$.
\end{theorem}

\begin{remark}[The veering correspondence] \label{rem_correspondence}
The correspondence between pseudo-Anosov flows and veering triangulations is initially due to Segerman and Schleimer in a series of papers \cite{schleimer2024loom, schleimer2019veering, SS3} and also establishes the existence of a pseudo-Anosov flow on $M$ associated to $V$. For us, however, it is more direct to apply \Cref{thm_AT5.1} in order to obtained the other properties needed for \Cref{prop:veering_th}.
\end{remark}

To explain the terminology ``ladderpole", and along the way specify our choice of slopes, we begin by describing a truncated model for $V$.  

\smallskip

The {\em truncated model} of $V$ is the 3-manifold with torus boundary components 
obtained by removing small neighborhoods of the cusps.  Alternatively, this can be obtained by gluing together truncated tetrahedra as in \Cref{fig:tet_ladder} (left), in place of the ideal tetrahedra.  By lifting, we obtain also a truncated model of $\wt V$ in which the cusps are now cylinders. 
This is because each point $p \in \kappa$ has infinite cyclic stabilizer by \Cref{prop:npfs}.
 The core curve of this cylinder projects to an essential curve on the corresponding torus cusp of $V$, giving us a well-defined slope $[s_T] \in H_1(T)$ for each torus cusp $T$ of $V$.   

Ladderpole curves are defined using the tiling of each cusp of $\wt V$ given by the edges of the truncated tetrahedra.  
This structure was first explained by Futer--Gu\'eritaud \cite[Section 2]{futer2013explicit}; our description is closer to that of  \cite[Section 5]{LMT20} and \cite[Section 2]{landry2019stable}.
Fix $p \in \kappa$ and consider the associated cusp $c_p$.  
Call the intersection of $c_p$ with a (tip of a) tetrahedron a {\em flat triangle} in $c_p$.
The edges of each flat triangle correspond to the three faces of the tetrahedron that meet $c_p$. Since the faces of $\wt V$ are each cooriented, this gives a coorientation on the sides of each flat triangle in $c_p$:  
the \emph{upward} flat triangles are exactly the ones meeting the top edge of the associated tetrahedron and the \emph{downward} flat triangles are the ones meeting its bottom edge. (Here, the flat triangle contains the endpoint of either the top or the bottom edge of its tetrahedron, and this marks the distinction between upward and downward.) Alternatively, upward flat triangles are exactly the ones for which two of its edges are cooriented out of the triangle.
See \Cref{fig:tet_ladder}.

\begin{figure}[htbp]
\begin{center}
\includegraphics[width = 11cm]{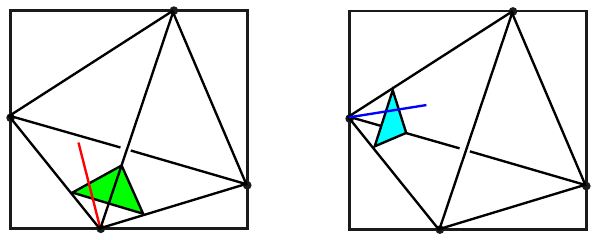}
\caption{Left: upward flat triangles (green) come from prongs in $\mc F^v$. Right: downward flat triangles (blue) come from prongs in $\mc F^h$.}
\label{fig:ladders}
\end{center}
\end{figure}

Translating this to the construction of $\wt V$ using maximal rectangles, the flat triangles in $c_p$ are precisely the ones associated to maximal rectangles of ${\mc P}$ that contain $p$ in their boundary. These are naturally divided into groups (called \emph{ladders}) corresponding to the prongs (i.e. half-leaves) of $\FF^{v/h}$ at $p$. For each prong $r$, define the ladder $L_r$ to be the union of flat triangles in $c_p$ of tetrahedra associated to maximal rectangles whose interiors meets $r$. By construction, the upward flat triangles of a tetrahedron $t_R$ are associated to the cusps along the horizontal sides of its maximal rectangle $R$ and the downward flat triangles of $t_R$ are associated to the cusps along the vertical sides of $R$. This translates into the fact, 
illustrated in \Cref{fig:ladders}, that
$r$ is in $\FF^v$ if and only if each flat triangle in $L_r$ is upward, as with the green triangles are in \Cref{fig:tet_ladder}. Otherwise, $r$ is  in $\FF^h$ and the flat triangles in $L_r$ are downward, as with blue triangles are in \Cref{fig:tet_ladder}. In this way, $c_p$ is the union of \emph{upward ladders} and \emph{downward ladders}, each an infinite strip in $c_p$ that meet along \emph{ladderpole lines}, the vertical lines in \Cref{fig:tet_ladder}.

Now, let $T$ be a torus cusp of $V$, where the upward and downward ladders project to a finite collection of annuli whose union is $T$.  The ladderpoles in any component of the preimage of $T$ in $\wt V$ project to a collection of curves $l_T$ in $T$. We call $\mc L_T = 1/2 \cdot [l_T] \in H_1(T)$ the associated \emph{ladderpole class}. (This is the same as the ladderpole class in \cite{AgolTsang} since they consider the core curves of the downward ladders in $T$.)

\begin{lemma} \label{lem:slope}
If $p \in {\mc P}$ is a $k$-prong in $\mc F^h$ (or $\mc F^v$), then $\langle s_T, \mc L_T \rangle = k$, where $T$ is the image in $V$ of the cusp $c_p$ in $\wt V$.
\end{lemma}

\begin{proof} 
Most of the proof has already been achieved by the discussion above:  
the number of upward (or downward) ladders is exactly the degree of the prong $p$ in the foliation $\mc F^h$ (or $\mc F^v$).  Because ladderpoles alternate upwards and downwards,  this is number is exactly half the number of ladderpoles in $c_p$.  
\end{proof}

\begin{figure}[htbp]
\begin{center}
\includegraphics[width = .9\textwidth]{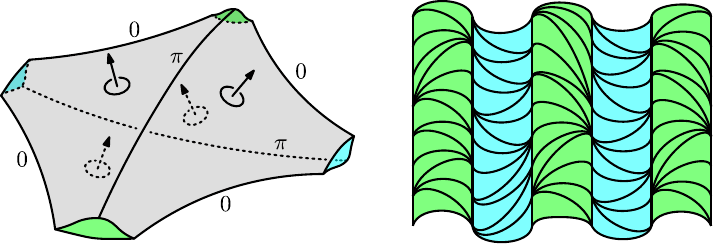}
\caption{Left: an ideal tetrahedra. Right: the pattern of upward and downward ladders along a cusp. Here the coorientation is upward. Figures from \cite{LMT21}.}
\label{fig:tet_ladder}
\end{center}
\end{figure}

We can now prove the main proposition
\begin{proof}[Proof of \Cref{prop:veering_th}]

Since $\mc P_\ast$ has no 1-pronged singularities, for each cusp $T$ we have 
$\langle [s_T], \mc L_T \rangle \ge 2$ by \Cref{lem:slope}.  Thus, we can apply Agol--Tsang (\Cref{thm_AT5.1}) and conclude that $M$ 
admits a transitive pseudo-Anosov flow $\varphi$ so that there are closed orbits $\wh \kappa_\phi = \{ \gamma_T \}$ isotopic to the cores of the filling solid tori with the property that each $\gamma_T$ is a $\langle [s_T], [l_T] \rangle$-prong, and $\varphi$ has no perfect fits relative to $\{\gamma_T\}$.

Now item $(3)$ is precisely \cite[Theorem 2.4.1]{tsang2023veering}.  

The claim in item $(1)$ is proved in \cite[Theorem 5.1]{LMT21} for any veering triangulation that is obtained by applying the {Agol--Gu\'eritaud construction} to the orbit space of a pseudo-Anosov flow. Since this is indeed the case by the (already established) item $(3)$, item $(1)$ follows.

Hence, it remains to prove item $(2)$. 
The inclusion $V \subset M$ from item (1) allows us to canonically identify $G$ with $\pi_1(M)$ as follows:
First note that the surjective homomorphism $\pi_1(V) \to \pi_1(M)$ induces an isomorphism $\pi_1(V)/ \langle\langle s_T  \rangle\rangle = \pi_1(M)$. Hence, if $\wt M$ denotes the universal cover, we have a $\pi_1(M)$--invariant embedding $\wt V \to \wt M$ lifting the fixed inclusion $V \subset M$. It now follows from basic covering space theory that the covering group of $\wt M \to M$, which is identified with $\pi_1(M)$, restricts to the covering group of $\wt V \to V$, which is $G$. 

This determines the isomorphisms $G \to \pi_1(M)$ that is equivariant with respect to the embedding $\wt V \to \wt M$, as required.
\end{proof}

Hence, we have constructed a closed manifold $M$ with $\pi_1(M) = G$ together with a transitive pseudo-Anosov flow $\varphi$. It remains to prove that the induced action $G \curvearrowright \orb_\varphi$ is conjugate to the original action $G \curvearrowright \mc P$, and this is the content of the next section.

\begin{remark}
One path to proving that the action $G \curvearrowright \orb_\varphi$ is conjugate to $G \curvearrowright \mc P$ is to use the theory of Anosov-like actions and the main result of \cite{BFM}. This would require first showing that  $G \curvearrowright \mc P$ is Anosov-like
 and then verifying that the two actions have the same set of elements fixing points in their respective planes. However, our approach here is to directly build the conjugacy and does not require much more effort. It also has the advantage of providing a self-contained dictionary between the flow and veering triangulation, which we think will be useful in future work. 
\end{remark}

\section{Recovering the action $G \curvearrowright \mc P$}
\label{sec:same_action}
The following proposition completes the proof of the main theorem. Informally, it states that a veering triangulation (when it exists) recovers the bifoliated plane, up to equivariant isomorphism. For the statement, we say that two actions are conjugate if there is a structure preserving, equivariant homeomorphism between their underlying spaces. 

\begin{proposition} 
\label{prop:recovering}
Suppose that $G$ acts on bifoliated planes $\mc P_1$ and $\mc P_2$. Let $\kappa_i \subset \mc P_i$ be a $G$--invariant, closed, discrete subset of $\mc P_i$ such that $\mc P_i$ has no perfect fits (or infinite strips) relative to $\kappa_i$, and let $\wt V_i$ be the associated veering triangulation.  If  the induced actions $G \curvearrowright \wt V_1$ and $G \curvearrowright \wt V_2$ are conjugate, then the actions of $G$ on $\mc P_1$ and $\mc P_2$ are conjugate.
\end{proposition}

To prove this, we show how to recover a bifoliated plane from a veering triangulation.  Several ideas appear in \cite[Section 6]{LMT21}, which is written in the (a priori) more restrictive setting of orbit spaces. 

\begin{remark} \label{rem_Segerman_schleimer}
The above result can be rephrased as a construction of $\mc P$ from $\wt V$. That veering triangulations determine bifoliated planes (in a slightly more restricted setting) also follows from the Segerman--Schleimer loom space construction. Our argument here using the dual graph ties more directly to the corresponding flow and is in the more general setting of our main theorem. 
\end{remark}

We start by working with a single plane; so fix $\mc P$ with an action of a group $G$, a discrete $G$-invariant set $\kappa \subset \mc P$ killing all perfect fits, 
and fix a veering action 
$G \curvearrowright \wt V$ obtained from the Agol--Gu\'eritaud  construction applied to an action $G \curvearrowright \mc P$ relative to $\kappa$, exactly as in \Cref{sec:veering}.  
The first lemma shows how to recover points of ${\mc P}$ from sequences of maximal rectangles; it is essentially the content of \cite[Fact 4.4]{LMT21}.  

\begin{lemma}[Maximal rectangle intersection] 
\label{lem_max}
If $(R_i)_{i \in \ZZ}$ is any sequence of maximal rectangles such that $R_{i+1}$ is above $R_i$, then $\bigcap_{i\ge 0} R_i$ is a single vertical segment of $R_0$ and $\bigcap_{i\le  0} R_i$ is a single horizontal segment of $R_0$; thus 
$\bigcap_{i \in \ZZ}(R_i)$ is a single point.
\end{lemma}

\begin{proof} 
Suppose that $\bigcap_{i\ge 0} R_i$ is not a single vertical segment in $R_0$. Then $\bigcap_{i\ge 0} R_i$ is a rectangle that can be extended vertically to a rectangle $Q$ whose horizontal boundary segments meets points of $\kappa$. This is because $\mc P$ has no perfect fits or infinite strips disjoint from $\kappa$. But then $Q$ lies above each rectangle $R_i$. This forces the $\kappa$-points along the horizontal boundary segments of the $R_i$ to accumulate in $\mc P$, contradicting the assumption that $\kappa$ is closed and discrete. The same argument gives that $\bigcap_{i \le 0} R_i$ is a single horizontal segment of $R_0$, which implies the lemma.
\end{proof} 

The next lemma says that sequences of rectangles that lie above/below each other can be coded (non-uniquely) by the  directed dual graph of $\wt V$.  
Recall, the {\em directed dual graph} is the graph $\wt \Gamma$ with a vertex for each tetrahedron $t_R$, where $R$ is a maximal rectangle, and a directed edge respecting the coorientation when two tetrahedra are glued along a face.  Thus, $\Gamma := \wt \Gamma /G \subset V = \wt V /G$ is the usual dual graph considered in \cite{LMT21}.

\begin{lemma} \label{lem_paths}
Let $\wt \Gamma$ be the directed dual graph of $\wt V$.  
A vertex $t_R$ has a directed path in $\wt \Gamma$ to $t_{R'}$ if and only if $R'$ lies above $R$ in $\mc P$.
\end{lemma}

\begin{proof}
First, each edge of $\wt \Gamma$ corresponds to a pair of maximal rectangles that share three points of $\kappa$,  and by definition the tetrahedron on top (according to the coorientation on their shared face) corresponds to the rectangle that lies above. Since ``above" is a transitive relation, whenever there is a direct path from $t_R$ to $t_{R'}$ in $\wt \Gamma$, the rectangle $R'$ lies above $R$.

Conversely, suppose that $R'$ lies above $R$. If we let $R_a, R_b$ be the two maximal rectangles corresponding to the two tetrahedra glued to the top faces of $t_R$, then $R'$ is either equal to or lies above one of $R_a, R_b$. In the former case we are done and in the latter we replace $R$ with the $R_a$ or $R_b$ below $R'$ and continue. If both $R_a$ and $R_b$ are below $R'$,  we just choose one.  This process must terminate by \Cref{lem_max} and the proof is complete.
\end{proof}

As an immediate consequence of \Cref{lem_max} and \Cref{lem_paths} we have: 
\begin{corollary}  \label{cor_defined}
Let $\mc L$ denote the set of bi-infinite directed paths  in $\wt \Gamma$.  There is a well defined map
$\Phi \colon \mc L \to \mc P$ given by
\[
l = (t_{R_i})_{i \in \ZZ} \mapsto \bigcap_i R_i
\]
\end{corollary} 
We will refer to elements of $\mc L$ as {\em dual lines}. 

Thus, in order to reconstruct $\mc P$ from $\wt V$, what we need is a description (purely in terms of $\wt V$ and $\wt \Gamma$) of which dual lines have the same image under $\Phi$.   This requires some set-up, and is the content of the next subsection.  

\subsection{Coding points in $\mc P$ via $\wt V$}
One way two elements of $\mc L$ can have the same image under $\Phi$ is if each always has a rectangle above (and below) each rectangle of the other.  To formalize this, we say $l_2 \in \mc L$ \emph{forward dominates} $l_1 \in \mc L$ if for every $t_1 \in l_1$, 
there is a $t_2 \in l_2$ such that there is a directed path from $t_1$ to $t_2$ in $\wt \Gamma$. If, in addition, $l_1$ forward dominates $l_2$, we say that $l_1$ and $l_2$ are \emph{forward equivalent}. 
Similarly, $l_2$ \emph{backwards dominates} $l_1$ if for every $t_1 \in l_1$, there is a $t_2 \in l_2$ such that there is a directed path from $t_2$ to $t_1$ in $\wt \Gamma$. The dual lines $l_1$ and $l_2$ are \emph{backward equivalent} if they backward dominate each other.
We say that dual lines are \emph{equivalent} if they are both forward and backward equivalent. 

\smallskip
For points of $\kappa$, or along leaves containing $\kappa$, additional difficulties arise because they can be limited on by sequences of rectangles in different quadrants.  For example, for any $p \in \kappa$ and quadrant $Q$ of $p$, 
one can make a sequence $R_i, i \in \mathbb{Z}$ with $\bigcap_i R_i = p$ and such that finite intersections $\bigcap_{-n < i<n} R_i$ lie in $Q$ for all $n>0$; and one can also make sequences where such intersections do not lie in a quadrant, but instead in one connected component of $\mc P \ssm (\FF^v(p))$ or of $\mc P \ssm (\FF^h(p))$.  
Our next goal is to describe how to recover this kind of multiplicity of codings from $\wt V$ and $\wt \Gamma$.  
 
As discussed in \Cref{sec:veering}, points of $\kappa$ correspond to cusps of $\wt V$, and 
given a cusp $c$ of $\wt V$, its induced triangulation is a union of upward and downward ladders.  Each upward ladder contains a unique dual line that we call the \emph{stable cusp line}. See \Cref{fig:tet_ladder}. Similarly, each downward ladder contains a unique dual line that we call the \emph{unstable cusp line}.  (These are called the stable/unstable branch lines in \cite{LMT21} since they correspond to lines of the branching locus of the stable/unstable branched surfaces in $\wt V$, respectively, however we  will not use this perspective here.) 

More generally, any path in $\wt V$ that is positively transverse to the faces of $\wt V$ determines a unique path in $\wt \Gamma$. If $c$ is a cusp of $\wt V$, the dual lines in $\wt \Gamma$ that are determined by bi-infinite positively transverse paths in a (punctured solid torus) neighborhood of $c$ are called the \emph{cusp lines} of $c$.

Put otherwise, the cusp lines of $c$ are exactly the stable/unstable cusp lines along the upward/downward ladders at $c$ (as defined above) together with `hybrid' dual lines consisting of half a stable cusp line and half an unstable cusp lines corresponding to adjacent upward and downward ladders on $c$. 

\begin{remark}
When $\mc P$ is the orbit space of a pseudo-Anosov flow, stable/unstable cusps lines arise from the flow as follows (see \cite[Lemma 3.12]{landry2025transverse}):
each upward ladder intersects a unique singular stable half-leaf along a line, 
and each downward ladder intersects a unique singular unstable half-leaf along a line, and these will match (as transverse paths in $\wt V$) the associated stable/unstable cusp lines.  
\end{remark}

Now, for a dual line $l = \{t_{R_i}\} \in \mc L$, define
\[
\mc F^v(l) = \bigcup_{n\ge 0}\bigcap_{i\ge n} R_i \quad \text{ and } \quad \mc F^h(l) =  \bigcup_{n \le 0}\bigcap_{i\le n} R_i.
\]

\begin{lemma} \label{lem:dominate}
With notation as above,
\begin{itemize}
\item If $l$ is forwards equivalent to a stable cusp line, then $\mc F^v(l)$ is the corresponding stable half-leaf of $\mc F^v$ though $\Phi(l)$. Otherwise, $\mc F^v(l)$ is a stable leaf face through $\Phi(l)$. Similar statements also hold for $\mc F^h(l)$.
\item If $l_2$ forward dominates $l_1$ then $\mc F^v(l_1) \subset \mc F^v(l_2)$. If $l_2$ backwards dominates $l_1$, then $\mc F^h(l_1) \subset \mc F^h(l_2)$
\end{itemize}
\end{lemma}

\begin{proof}
From \Cref{lem_max}, we have that intersections of the form $\bigcap_{i\ge 0} R_i$ are vertical segments of $R_0$. Hence, $\mc F^v(l)$ is an increasing union of vertical segments. These segments limit to a leaf face unless some horizontal side of the associated rectangles eventually all contain a single point of $\kappa$. However, this happens if and only if $l$ is forward dominated by a stable cusp line.   This proves the first item. 

The second is a direct consequence of the definition of forward (resp. backwards) dominate and the fact that $\bigcap_{i\ge n} R_i$ are vertical segments  (respectively, that $\bigcap_{i \leq n} R_i$ are horizontal segments). 
\end{proof}

\begin{proposition} \label{th:veering_to_plane}
The map $\Phi$  has the following properties:  
\begin{enumerate}
\item $\Phi$ is $G$--equivariant and surjective.
\item If $\mc L_R \subset \mc L$ is the set of all dual lines that contains the tetrahedron $t_R$, for some maximal rectangle $R$, then $\Phi(\mc L_R) = R$.
\item For each point $x \in \mc P$ not contained in a leaf of a point of $\kappa$, $\Phi^{-1}(x)$ is exactly one class of equivalent dual lines.
\item For each point $x \in \mc P \ssm \kappa$ contained in a stable (resp. unstable) half-leaf at $\kappa$, $\Phi^{-1}(x)$ is exactly one class of dual lines which forward (resp. backward) dominate the corresponding stable (resp. unstable) cusp line 
and are backward (resp. forward) equivalent.
\item For each point $x \in \kappa$, $\Phi^{-1}(x)$ is exactly the set of cusp lines of the cusp $c$ associated to $x$.
\end{enumerate}
\end{proposition}

We note that items $(3)$, $(4)$, and $(5)$ exactly characterize the point preimages of the map $\Phi \colon \mc L \to \mc P$.

\begin{proof}
We begin with the first item. Since 
\[
\Phi(g(t_{R_{i}})) = \Phi((t_{gR_{i}})) = g \left (\bigcap_i R_i \right ),
\]
the map is equivariant. 

To see surjectivity, let $p \in \mc P$ and pick an arbitrary maximal rectangle $R$ containing $p$, possibly in its boundary. (The are many choices here, and we will soon have to reckon with the possibilities.) Set $R_0 = R$. Among the two taller maximal rectangles that share a face with $R$, at least one (though possibly both) contains $p$. Let $R_1$ be such a rectangle. Similarly, let $R_{-1}$ be a 
maximal rectangle below $R$ that shares a face with $R$ and which also contains $p$. Continuing in this way, we produce a dual line $l = (t_{R_i})_{i\in \ZZ}$ so that by construction $\Phi(l) = \bigcap_i R_i = p$.  This also proves item 2. 

For the last three items, first note that if $\ell_1,\ell_2 \in \mc L$ are such that one forward dominates the other and one backwards dominates the other, then $\Phi(l_1) = \Phi(l_2)$, since \Cref{lem:dominate} shows they determine the same (half-)leaves which intersect in a single point. 

If $l_1,l_2 \in \Phi^{-1}(x)$, then all the maximal rectangles along $l_1$ or $l_2$ contain $x$. By \Cref{lem:dominate}, $\mc F^v(l_1)$ and $\mc F^v(l_2)$ are each either stable leaf faces through $x$ or a stable half-leaf of $\mc F^v$, which occurs if and only if the dual line is a stable cusp line. Hence, either $l_1,l_2$ are forward equivalent or they both dominate the associate stable cusp line. Here, we are using that fact that if $\mc F^v(l)$ intersects a maximal rectangle $R$, then maximal rectangles along $l$ are eventually above $R$ in the forward direction. Since the same reasoning holds in the backward direction (using the horizontal/unstable foliation), this establishes items (3) and (4).

Finally, item (5) follows from the picture of maximal rectangles around a point of $\kappa$.
\end{proof}

With the work above, we can now quickly prove 
 \Cref{prop:recovering}
 
\begin{proof}[Proof of \Cref{prop:recovering}]
Suppose $G$ acts on bifoliated planes $\mc P_1$ and $\mc P_2$, preserving sets $\kappa_1, \kappa_2$  such that $\mc P_i$ has no perfect fits or infinite strips relative to $\kappa_i$. Let $\wt V_i$ be the induced veering triangulations, and suppose the induce actions of $G$ on the $\wt V_i$ are conjugate by a simplical homeomorphism that preserves the coorientation on each face.

Let $\Phi_i  \colon \mc L_i \to \mc P_i$ be the $G$-equivariant maps defined in \Cref{cor_defined}.  
Since the actions on the triangulations $\wt V_i$ are conjugate, the associated dual graphs $\wt \Gamma_i$ are naturally isomorphic via an equivariant isomorphism $J\colon \wt \Gamma_1 \to \wt \Gamma_2$.  In particular, $J$ sends (stable/unstable) cusp lines to (stable/unstable) cusp lines and preserves the properties of forward/backward dominating. Hence, items $(3)$, $(4)$, and $(5)$ from \Cref{th:veering_to_plane}, which characterize the point preimages of $\Phi_i$, imply that
two dual lines $l, l'$ in $\wt \Gamma_1$ have the same image under $\Phi_1$ if and only if $J(l)$ and $J(l')$ have the same image under $\Phi_2$. 

From this it follows that $\Phi_1 \circ \Phi_2^{-1} \colon \mc P_2 \to \mc P_1$ is a well-defined, $G$--equivariant bijection.  Moreover, by item $(2)$ of  \Cref{th:veering_to_plane}, it sends maximal rectangles to maximal rectangles. 
Since intersections of interiors of maximal rectangles (or finite unions of maximal rectangles around points of $\kappa$) generate the topology of $\mc P$, this map and its inverse are continuous. Thus $\Phi_1 \circ \Phi_2^{-1} $ is a homeomorphism conjugating the action of $G$. This completes the proof.
\end{proof}

Combining \Cref{prop:veering_th} and \Cref{prop:recovering} now complete the proof of the `if' direction of main theorem (\Cref{th:main}). The `only if' direction was proven in \Cref{prop:only_if}.

\begin{remark}
Using veering triangulations gives another approach to the `only if' direction of \Cref{th:main}, which we proved using Markov partitions.
Indeed, when the closed oriented manifold $M$ supports a transitive pseudo-Anosov flow $\varphi$, then as we saw in \Cref{sec:pfits}, there is a group invariant, closed, discrete set of points $\kappa$ in $\orb_\phi$ such that $\orb_\phi$ has no perfect fits relative to $\kappa$. Then, as in \Cref{sec:veering}, we can consider the set of maximal rectangles whose interiors are disjoint from $\kappa$.  As in the proof of \Cref{prop:its_veering}, this set is finite up to the group action and so we can let $\mc R$ be a finite set of representatives under the action. 

Now it is clear that $\mc R$ satisfies that finite rectangle condition since, just as in \Cref{lem_paths}, for each $R \in \mc R$, $R$ is covered by the maximal rectangles for the two tetrahedra directly above $t_R$ and each of these rectangles is itself above $R$. Similarly, $R$ is covered by the maximal rectangles for the two tetrahedra directed below $t_R$ and these rectangles are below $R$.
\end{remark}

\bibliographystyle{alpha}
\bibliography{orbit.bib}

\end{document}